\documentclass[12pt,american]{article}
\usepackage[T1]{fontenc}
\usepackage[latin9]{inputenc}
\usepackage[a4paper]{geometry}
\usepackage[active]{srcltx}
\usepackage{xcolor}
\usepackage{babel}
\usepackage{verbatim}
\usepackage{enumitem}
\usepackage{amsmath}
\usepackage{amsthm}
\PassOptionsToPackage{normalem}{ulem}
\usepackage{ulem}
\usepackage[unicode=true,pdfusetitle,
 bookmarks=true,bookmarksnumbered=false,bookmarksopen=false,
 breaklinks=false,pdfborder={0 0 0},pdfborderstyle={},backref=false,colorlinks=true]
 {hyperref}
\usepackage{breakurl}

\makeatletter

\providecolor{lyxadded}{rgb}{0,0,1}
\providecolor{lyxdeleted}{rgb}{1,0,0}
\DeclareRobustCommand{\lyxsout}[1]{\ifx\\#1\else\sout{#1}\fi}

\theoremstyle{plain}
\newtheorem{thm}{\protect\theoremname}[section]
  \theoremstyle{plain}
  \newtheorem{prop}[thm]{\protect\propositionname}
  \theoremstyle{plain}
  \newtheorem{question}[thm]{\protect\questionname}
  \theoremstyle{plain}
  \newtheorem{lem}[thm]{\protect\lemmaname}
  \theoremstyle{plain}
  \newtheorem{cor}[thm]{\protect\corollaryname}
  \theoremstyle{remark}
  \newtheorem{rem}[thm]{\protect\remarkname}
  \theoremstyle{remark}
  \newtheorem{claim}[thm]{\protect\claimname}

\date{}
\usepackage{calc}

\providecommand\phantomsection{}

\makeatletter
\newenvironment{boldproof}[1][\proofname] {\par\pushQED{\qed}\normalfont\topsep6\p@\@plus6\p@\relax\trivlist\item[\hskip\labelsep\bfseries#1\@addpunct{.}]\ignorespaces}{\popQED\endtrivlist\@endpefalse}
\makeatother
\usepackage{fullpage}
\usepackage{tikz}\usetikzlibrary{arrows}

\makeatother

  \providecommand{\claimname}{Claim}
  \providecommand{\corollaryname}{Corollary}
  \providecommand{\lemmaname}{Lemma}
  \providecommand{\propositionname}{Proposition}
  \providecommand{\questionname}{Question}
  \providecommand{\remarkname}{Remark}
\providecommand{\theoremname}{Theorem}

\begin{document}
\global\long\def\C{\mathcal{C}}
\global\long\def\F{\mathcal{F}}
\global\long\def\h{\mathcal{\mathcal{H}}}
\global\long\def\Hanm{\h(a,n,m)}
\global\long\def\o{\mathcal{\mathcal{O}}}
\global\long\def\Oanm{\o(a,n,m)}
\global\long\def\Otanm{\tilde{\o}(a,n,m)}
\global\long\def\T{\mathcal{\mathcal{T}}}
\global\long\def\naturals{\mathbf{N}}
\global\long\def\tbd{{\color{red}??}}
\global\long\def\ex{\operatorname{ex}}
\global\long\def\E{\operatorname{E}}
\global\long\def\red{\mathrm{red}}
\global\long\def\blue{\mathrm{blue}}
\global\long\def\abs#1{\left|#1\right|}

\title{On subgraphs of $C_{2k}$-free graphs and a problem of Kühn and Osthus}

\author{D\'aniel Gr\'osz\thanks{Department of Mathematics, University of Pisa. e-mail: \protect\href{mailto:groszdanielpub@gmail.com}{groszdanielpub@gmail.com}}
\and  Abhishek Methuku\thanks{Department of Mathematics, Central European University, Budapest,
Hungary. e-mail: \protect\href{mailto:abhishekmethuku@gmail.com}{abhishekmethuku@gmail.com}} \and  Casey Tompkins\thanks{Alfr\'ed R\'enyi Institute of Mathematics, Hungarian Academy of
Sciences. e-mail: \protect\href{mailto:ctompkins496@gmail.com}{ctompkins496@gmail.com}}}
\maketitle
\begin{abstract}
Let $c$ denote the largest constant such that every $C_{6}$-free
graph $G$ contains a bipartite and $C_{4}$-free subgraph having
$c$ fraction of edges of $G$. Gy\H ori \emph{et al}.\ showed that
$\frac{3}{8}\le c\le\frac{2}{5}$. We prove that $c=\frac{3}{8}$.
More generally, we show that for any $\varepsilon>0$, and any integer
$k\ge2$, there is a $C_{2k}$-free graph $G_{1}$ which does not
contain a bipartite subgraph of girth greater than $2k$ with more
than $\left(1-\frac{1}{2^{2k-2}}\right)\frac{2}{2k-1}(1+\varepsilon)$
fraction of the edges of $G_{1}$. There also exists a $C_{2k}$-free
graph $G_{2}$ which does not contain a bipartite and $C_{4}$-free
subgraph with more than $\left(1-\frac{1}{2^{k-1}}\right)\frac{1}{k-1}(1+\varepsilon)$
fraction of the edges of $G_{2}$. 

One of our proofs uses the following statement, which we prove using
probabilistic ideas, generalizing a theorem of Erd\H os: For any
$\varepsilon>0$, and any integers $a$, $b$, $k\ge2$, there exists
an $a$-uniform hypergraph $H$ of girth greater than $k$ which does
not contain any $b$-colorable subhypergraph with more than $\left(1-\frac{1}{b^{a-1}}\right)\left(1+\varepsilon\right)$
fraction of the hyperedges of $H$. We also prove further generalizations
of this theorem.\smallskip{}

In addition, we give a new and very short proof of a result of Kühn
and Osthus, which states that every bipartite $C_{2k}$-free graph
$G$ contains a $C_{4}$-free subgraph with at least $1/(k-1)$ fraction
of the edges of $G$. We also answer a question of Kühn and Osthus
about $C_{2k}$-free graphs obtained by pasting together $C_{2l}$'s
(with $k>l\ge3$).
\end{abstract}

\section{Introduction}

\begin{comment}
Definitions of Berge-cycle, linear hypergraph, girth. We consider
Berge-cycles of length 2 when talking about girth, and no 2-cycle
implies the hypergraph is linear.
\end{comment}

Let $e(H)$ denote the number of (hyper)edges in a (hyper)graph $H$.
For a family of graphs $\F$, let $\ex(n,\F)$ denote the maximum
number of edges in an $n$-vertex graph which does not contain any
$\text{\ensuremath{F\ensuremath{\in\F}}}$ as a subgraph. (In the
case when $\F=\{F\}$, we write simply $\ex(n,F)$.) The girth of
a graph is defined as the length of a shortest cycle if it exists,
and infinity otherwise. In \cite{Gyori:1997}, Gy\H ori proved that
every bipartite, $C_{6}$-free graph contains a $C_{4}$-free subgraph
with at least half as many edges. Later K\"uhn and Osthus \cite{Kuhn:2005}
generalized this result by showing
\begin{thm}[K\"uhn and Osthus \cite{Kuhn:2005}]
\label{thm:Kuhn_Osthus}Let $k\ge3$ be an integer and $G$ a $C_{2k}$-free
bipartite graph. Then $G$ contains a $C_{4}$-free subgraph $H$
with $e(H)\ge\frac{e(G)}{k-1}$.
\end{thm}

In Section \ref{sec:simpleproof} we give a new short proof of their
result. The complete bipartite graphs $K_{k-1,m}$ (for large enough
$m$) show that the factor $\frac{1}{k-1}$ cannot be replaced by
anything larger (see Proposition $5$ in \cite{Kuhn:2005}).

F\"uredi, Naor and Verstra\"ete \cite{Furedi:2006} gave another
generalization of Gy\H ori's theorem by showing that every $C_{6}$-free
graph $G$ has a subgraph of girth larger than $4$ with at least
half as many edges as $G$. Again, $K_{2,m}$ shows that this factor
cannot be improved. It follows that $\ex(n,C_{6})\le2\cdot\ex(n,\{C_{4},C_{6}\})$.
Since any graph has a bipartite subgraph with at least half as many
edges, Theorem \ref{thm:Kuhn_Osthus} shows that $\ex(n,C_{2k})\le2(k-1)\cdot\ex(n,\{C_{4},C_{2k}\})$.
These results confirm special cases of the compactness conjecture
of Erd\H os and Simonovits \cite{erdHos1982compactness} which states
that for every finite family $\F$ of graphs, there exists an $F\in\F$
such that $\ex(n,F)=O(\ex(n,\F))$.

Since any $C_{6}$-free graph contains a bipartite subgraph with at
least half as many edges, using any of the results above it is easy
to show that any $C_{6}$-free graph $G$ has a bipartite, $C_{4}$-free
subgraph with at least $\frac{1}{4}$ of the edges of $G$. Gy\H ori,
Kensell and Tompkins \cite{gyHori2015making} improved this factor
by showing that 
\begin{thm}[Gy\H ori, Kensell and Tompkins \cite{gyHori2015making}]
If $c$ is the largest constant such that every $C_{6}$-free graph
$G$ contains a $C_{4}$-free and bipartite subgraph $B$ with $e(B)\ge c\cdot e(G)$,
then $\frac{3}{8}\le c\le\frac{2}{5}$.
\end{thm}

The complete graph $K_{5}$ (as well as a graph consisting of vertex
disjoint $K_{5}$'s) gives that $c\le\frac{2}{5}$. To show that $\frac{3}{8}\le c$
they use a probabilistic deletion procedure where they first randomly
two-color the vertices, and then delete some additional edges carefully
in order to remove the remaining $C_{4}$'s. In this paper we show
that $c=\frac{3}{8}$. In fact, we prove the following two general
results; putting $k=3$ in either of the statements below gives that
$c=\frac{3}{8}$. To prove these theorems we will construct graphs
by replacing the hyperedges of certain (probabilistically constructed)
hypergraphs with fixed small graphs.
\begin{thm}
\label{thm:max_c}For any $\varepsilon>0$, and any integer $k\ge2$,
there is a $C_{2k}$-free graph $G$ which does not contain a bipartite
subgraph of girth greater than $2k$ with more than $\left(1-\frac{1}{2^{2k-2}}\right)\frac{2}{2k-1}e(G)(1+\varepsilon)$
edges.
\end{thm}

Note that $K_{2k-1}$ is $C_{2k}$-free, and the only subgraphs with
girth greater than $2k$ are forests, giving an upper bound of $\frac{2}{2k-1}e(G)(1+\varepsilon)$.
The factor $1-\frac{1}{2^{2k-2}}$ is the probability that a random
two-coloring of $K_{2k-1}$ is not monochromatic.
\begin{thm}
\label{thm:C_4-only}For any $\varepsilon>0$, and any integer $k\ge2$,
there is a $C_{2k}$-free graph $G$ which does not contain a bipartite
and $C_{4}$-free subgraph with more than $\left(1-\frac{1}{2^{k-1}}\right)\frac{1}{k-1}e(G)(1+\varepsilon)$
edges.
\end{thm}

Theorem \ref{thm:C_4-only} improves the upper bound of $\frac{1}{k-1}e(G)(1+\varepsilon)$,
which is given by the complete bipartite graphs $K_{k-1,m}$. Take
a random bipartition of the vertices of $K_{k-1,m}$ and consider
the bipartite subgraph $B$ between the colour classes of this bipartition.
The factor $\left(1-\frac{1}{2^{k-1}}\right)\frac{1}{k-1}$ in the
above theorem is the limit of the expected value of the fraction of
edges of $K_{k-1,m}$ in the biggest $C_{4}$-free subgraph of $B$
as $m\rightarrow\infty$. (Note that because any graph has a bipartite
subgraph with at least half of its edges, Theorem \ref{thm:Kuhn_Osthus}
implies that every $C_{2k}$-free graph contains a bipartite and $C_{4}$-free
subgraph with at least $\frac{1}{2(k-1)}$ fraction of its edges.)
Interestingly, our proofs use theorems about hypergraphs that are
generalizations of the following theorem of Erd\H os \cite{erdos1967grafok}.

Every graph $G$ has a bipartite subgraph with at least $\frac{1}{2}$
as many edges as $G$, and the complete graph $K_{n}$ shows that
the factor $\frac{1}{2}$ cannot be improved. Interestingly, Erd\H os
showed that even if one requires girth to be large, the factor $\frac{1}{2}$
still cannot be improved. More precisely,%
\begin{comment}
That paper is interesting I hadn't seen it. It seems their goal is
to improve the lower order terms of the maxcut, so for C\_6-free we
actually can say there is a bipartite subgraph with e(G)/2 + C e(G)\textasciicircum{}(7/8)
edges. Yes, we should look at these type of papers, and should we
cite them? I dont know how relevant they are to our paper...
\end{comment}
\begin{thm}[Erd\H os \cite{erdos1967grafok}]
\label{thm:Erdos} For any $\varepsilon>0$, and any integer $k\ge2$,
there exists a graph $G$ with girth greater than $k$ which does
not contain a bipartite subgraph with more than $\frac{1}{2}e(G)\left(1+\varepsilon\right)$
edges.
\end{thm}

In Section \ref{sec:Hypergraph-lemmas}, we prove a series of lemmas
about hypergraphs which are broad generalizations of Theorem \ref{thm:Erdos},
and which may be of independent interest. These lemmas have the theme
that for most hypergraphs, every fixed coloring behaves like a random
coloring with color classes of the same sizes as in the fixed coloring.
Our proof of Theorem \ref{thm:C_4-only} uses these general lemmas
directly. The proof of Theorem \ref{thm:max_c} uses a more direct
analogue of the above statement for hypergraphs: Theorem \ref{thm:high girth small p},
which we present below. We will prove Theorem \ref{thm:high girth small p}
from the more general lemmas.

A Berge-cycle of length $l$ in a hypergraph $H$ is a subhypergraph
consisting of $l\ge2$ distinct hyperedges $\text{\ensuremath{e_{1}},}\ldots,e_{l}$
and containing $l$ distinct vertices $v_{1},\ldots,v_{l}$ (called
its \emph{defining} vertices), such that $v_{i}\in e_{i}\cap e_{i+1}$,
$i=1,\ldots,l$, where addition in the indices is taken modulo $l$.
The girth of a hypergraph $H$ is the length of a shortest Berge-cycle
if it exists, and infinity otherwise. (Note that having girth greater
than 2 implies that no two hyperedges share more than one vertex.)
A hypergraph is $b$-colorable if there is a coloring of its vertices
using $b$ colors so that none of its hyperedges are monochromatic.
Erd\H os and Hajnal \cite{erdos-hajnal} showed the existence of
hypergraphs of any uniformity, arbitrarily high girth and arbitrarily
high chromatic number. Lovász \cite{lovasz-constructive} gave a constructive
proof for this; several newer proofs exist as well. The following
simple proposition is easy to see. We include its proof for completeness.
\begin{prop}
\label{prop:random_coloring}For any integers $a,b\ge2$, every $a$-uniform
hypergraph $H$ contains a $b$-colorable subhypergraph with at least
$\left(1-\frac{1}{b^{a-1}}\right)e(H)$ hyperedges.
\end{prop}

\begin{proof}
Color each vertex of $H$ randomly and independently, using $b$ colors
with equal probability. For each hyperedge $f$ of $H$, the probability
that $f$ is monochromatic is $\frac{b}{b^{a}}=\frac{1}{b^{a-1}}$.
Therefore, the expected number of monochromatic hyperedges in $H$
is $\frac{e(H)}{b^{a-1}}$. So there exists a coloring of the vertices
of $H$ such that there are at most $\frac{e(H)}{b^{a-1}}$ monochromatic
hyperedges in that coloring. Thus, the subhypergraph of $H$ consisting
of all the non-monochromatic hyperedges of $H$ contains at least
$\left(1-\frac{1}{b^{a-1}}\right)e(H)$ hyperedges and is $b$-colorable,
as desired.
\end{proof}
Again the complete $a$-uniform hypergraph shows that the factor $\left(1-\frac{1}{b^{a-1}}\right)$
cannot be improved in the above proposition. We show that (as in case
of graphs), this factor cannot be improved even if one requires the
girth to be large.
\begin{thm}
\label{thm:high girth small p}For any $\varepsilon>0$, and any integers
$a,b,k\ge2$, there exists an $a$-uniform hypergraph $H$ of girth
more than $k$ which does not contain a $b$-colorable subhypergraph
with more than $\left(1-\frac{1}{b^{a-1}}\right)e(H)\left(1+\varepsilon\right)$
hyperedges.

\begin{comment}
TO DO: high chromatic number graphs (there were two papers, a probabilistic
and a constructive one?) RODL-NESETRIL = probabilistic, LOVASZ =constructive.
Actually, there was a constructive one for the existence of a high
girth, hard to make bipartite graph (not hypergraph) too, wasn't there?
(Alon?)
\end{comment}

\begin{comment}
Perhaps mentions about other related problems/motivations: Turan numbers;
making \_\_\_-free graphs \_\_\_-free and/or bipartite.
\end{comment}
\end{thm}

Clearly, letting $a=b=2$ in the above theorem, we get Theorem \ref{thm:Erdos}.
The hypergraph lemmas in Section \ref{sec:Hypergraph-lemmas} can
be used to prove statements analogous to Theorem \ref{thm:high girth small p}
with different notions of colorability. As an example application,
we will prove the analogous Proposition \ref{prop:rainbow} about
strong (or rainbow) colorable subhypergraphs. More generally, a graph
$G$ is called $H$-colorable (where $H$ is a fixed graph) if there
is a homomorphism $G\rightarrow H$. Our Lemma \ref{lem:subhypergraph}
can be said to generalize the notion of $H$-coloring to hypergraphs,
and allow for proving statements similar to Theorem \ref{thm:high girth small p}
for $H$-colorability or analogous hypergraph conditions.

In Section \ref{sec:pastings}, we answer a question of Kühn and Osthus
in \cite{Kuhn:2005}. A graph is said to be \emph{pasted together}
from $C_{2l}$'s if it can be obtained from a $C_{2l}$ by successively
adding new $C_{2l}$'s which have at least one edge in common with
the previous ones.
\begin{question}[Kühn, Osthus \cite{Kuhn:2005}]
\label{question} Given integers $k>l\ge2$, does there always exist
a number $d=d(k)$ such that every $C_{2k}$-free graph which is pasted
together from $C_{2l}$'s has average degree at most $d$? 
\end{question}

Kühn and Osthus show in \cite{Kuhn:2005} that an affirmative answer
to the above question, even when restricted to bipartite graphs, would
imply that any $C_{2k}$-free graph $G$ contains a $C_{2l}$-free
subgraph containing a constant fraction of the edges of $G$. They
gave a positive answer to the question when $l=2$ and the graph is
bipartite: they showed that if $k\ge3$ is an integer and $G$ is
a bipartite $C_{2k}$-free graph which is obtained by pasting together
$C_{4}$'s, then the average degree of $G$ is at most $16k$.

We answer Question \ref{question} negatively by showing two different
pastings of $C_{6}$'s to form a $C_{8}$-free graph with high average
degree. These two examples show (in two very different ways) that
many $C_{6}$'s can be packed into a graph while still keeping it
$C_{8}$-free. We will show that the first example can be easily generalized
to any pair $k,l$ with $k>l\ge3$, showing that $l=2$ is the only
case when any $C_{2k}$-free graph obtained by pasting together $C_{2l}$'s
has average degree bounded by a constant $d=d(k)$.

\emph{Our paper is organized as follows:} In Section \ref{sec:simpleproof},
we give a short proof of Theorem \ref{thm:Kuhn_Osthus}. In Section
\ref{sec:Hypergraph-lemmas}, we prove a series of hypergraph lemmas
and Theorem \ref{thm:high girth small p}. Our proofs in Section \ref{sec:Hypergraph-lemmas}
use counting arguments and probabilistic ideas very similar to Erd\H os's
proof. In Section \ref{sec:Proof-of-MAINTheorem} we prove Theorem
\ref{thm:max_c} and Theorem \ref{thm:C_4-only}. In Section \ref{sec:pastings},
we give two examples of pasting together $C_{6}$'s to form a $C_{8}$-free
graph with high average degree, answering Question \ref{question}.

\section{A simple proof \label{sec:simpleproof} of a theorem of K\" uhn
and Osthus (Theorem~\ref{thm:Kuhn_Osthus})}
\begin{proof}[Proof of Theorem \ref{thm:Kuhn_Osthus}]
Let $G$ be a $C_{2k}$-free bipartite graph with color classes $A:=\{a_{1},\allowbreak a_{2},\ldots,a_{m}\}$
and $B:=\{b_{1},b_{2},\ldots,b_{n}\}$ for some $m,n\ge1$. Order
the vertices in $A$ and $B$ as $a_{1}<a_{2}<\ldots<a_{m}$ and $b_{1}<b_{2}<\ldots<b_{n}$
respectively. An edge $ab\in E(G)$ with $a\in A$ and $b\in B$ is
denoted by the ordered pair $(a,b)$.

We define a partial order $\mathcal{{P}}=(E(G),\le_{p})$ on the edge
set of $G$ as follows. For any two edges $(a,b),(a',b')\in E(G)$,
we say that $(a,b)\le_{p}(a',b')$ if and only if there exists an
integer $r\ge1$ and edges $(p_{i},q_{i})\in E(G)$, $i=1,2,\ldots,r$
such that $a=p_{1},b=q_{1}$ and $a'=p_{r},b'=q_{r}$, and the following
conditions hold: $p_{i}<p_{i+1}$ and $q_{i}<q_{i+1}$ and the vertices
$p_{i},q_{i},p_{i+1},q_{i+1}$ induce a $C_{4}$ for all $1\le i\le r-1$.

It is easy to see that %
\begin{comment}
Made a figure 
\end{comment}
if there is a chain of length $k$ in $\mathcal{P}$ then $G$ contains
a cycle of length $2k$, a contradiction (see Figure \ref{chain}).
So the length of a longest chain in $\mathcal{P}$ is at most $k-1$
which implies that the size of a largest antichain in $\mathcal{P}$
is at least $\frac{1}{k-1}\abs{E(G)}$ by Mirsky's theorem \cite{mirsky1971dual}.
Since $G$ is bipartite, any $C_{4}$ in $G$ contains two edges $(p_{1},q_{1})$,
$(p_{2},q_{2})\in E(G)$ such that $(p_{1},q_{1})<_{p}(p_{2},q_{2})$,
so the subgraph $H$ of $G$ consisting of the edges in this largest
antichain is $C_{4}$-free, completing the proof of the theorem. 
\end{proof}
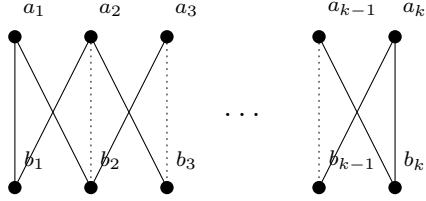
\begin{figure}
\begin{centering}
\begin{comment}
Position of the dots fixed manually 
\end{comment}
\begin{tikzpicture}[line cap=round,line join=round,>=triangle 45,x=1.0cm,y=1.0cm]
\begin{comment}
\clip(-4.3,-3.66) rectangle (9.16,6.3);
\end{comment}
\draw [dotted] (-1.6,2.06)-- (-1.6,0.06);
\draw (-2.6,2.06)-- (-2.6,0.06);
\draw [dotted] (-0.6,2.06)-- (-0.6,0.06);
\draw [dotted] (1.4,2.06)-- (1.4,0.06);
\draw (2.4,2.06)-- (2.4,0.06);
\draw (-2.6,0.06)-- (-1.6,2.06);
\draw (-2.6,2.06)-- (-1.6,0.06);
\draw (-1.6,0.06)-- (-0.6,2.06);
\draw (-1.6,2.06)-- (-0.6,0.06);
\draw (1.4,0.06)-- (2.4,2.06);
\draw (1.4,2.06)-- (2.4,0.06);
\draw (0,1.2) node[anchor=north west] {$\ldots$};
\begin{scriptsize}
\fill [color=black] (-2.6,2.06) circle (2.5pt);
\draw[color=black] (-2.34,2.42) node {$a_1$};
\fill [color=black] (-1.6,2.06) circle (2.5pt);
\draw[color=black] (-1.34,2.42) node {$a_2$};
\fill [color=black] (-0.6,2.06) circle (2.5pt);
\draw[color=black] (-0.34,2.42) node {$a_3$};
\fill [color=black] (1.4,2.06) circle (2.5pt);
\draw[color=black] (1.84,2.42) node {$a_{k-1}$};
\fill [color=black] (2.4,2.06) circle (2.5pt);
\draw[color=black] (2.66,2.42) node {$a_k$};
\fill [color=black] (-2.6,0.06) circle (2.5pt);
\draw[color=black] (-2.34,0.42) node {$b_1$};
\fill [color=black] (-1.6,0.06) circle (2.5pt);
\draw[color=black] (-1.34,0.42) node {$b_2$};
\fill [color=black] (-0.6,0.06) circle (2.5pt);
\draw[color=black] (-0.34,0.42) node {$b_3$};
\fill [color=black] (1.4,0.06) circle (2.5pt);
\draw[color=black] (1.84,0.42) node {$b_{k-1}$};
\fill [color=black] (2.4,0.06) circle (2.5pt);
\draw[color=black] (2.66,0.42) node {$b_k$};
\end{scriptsize}
\end{tikzpicture}
\par\end{centering}
\caption{The solid edges form a $C_{2k}$}
\label{chain}
\end{figure}

\section{Hypergraph lemmas\label{sec:Hypergraph-lemmas} and proof of Theorem
\ref{thm:high girth small p}}

Let $\Hanm$ denote the family of all $a$-uniform hypergraphs with
$n$ vertices and $m$ hyperedges for some $a\geq2$. $\left|\Hanm\right|={{n \choose a} \choose m}$.
Given a coloring $C:[n]\rightarrow[b]$ of the vertex set $[n]$ with
$b$ colors (with $b\geq2$), let $n_{j}^{C}$ be the number of vertices
of color $j$. The multiset of the colors of the vertices of a hyperedge
$e$ (with the multiplicity with which they occur in $e$) is called
the \emph{color multiset of $e$ (with respect to $C$)}. For an $a$-element
multiset of colors $T$, let $p^{C}(T)$ be the probability that the
color multiset of a random hyperedge of the complete $a$-uniform
hypergraph on $n$ vertices, with the coloring $C$, is $T$. (Note
that in this paper, when we mention a coloring, we mean an arbitrary
coloring of the vertex set, not necessarily a proper coloring of a
hypergraph, unless indicated.)
\begin{prop}
\label{prop:p-asymptotic}For $n\rightarrow\infty$, asymptotically
\[
p^{C}(T)=\frac{\prod_{j=1}^{b}\binom{n_{j}^{C}}{I_{T}(j)}}{\binom{n}{a}}\sim\frac{\prod_{j=1}^{b}\left(n_{j}^{C}\right)^{I_{T}(j)}}{n^{a}}\cdot\frac{a!}{\prod_{j=1}^{b}I_{T}(j)!}
\]
 where $I_{T}(j)$ denotes the multiplicity of $j$ in the multiset
$T$.
\end{prop}

We will also use the following tail bound on the binomial and the
hypergeometric distributions. Hoeffding proves this bound in a more
general setting, see Section 2 in \cite{hypergeometric_tail} for
the binomial distribution and Section 6 for the hypergeometric distribution.
If a random variable $X$ has binomial distribution with $m$ trials
and success probability $p$, we write $X\sim\textnormal{Binomial}(m,p)$.
If $X$ has hypergeometric distribution with a population of size
$N$ containing $pN$ successes, and with $m$ draws, we write $X\sim\textnormal{Hypergeometric}(pN,N,m)$.
\begin{prop}
\label{prop:tail}Let $m,N\in\naturals$ and $p,\varepsilon\in[0,1]$,
and let $X$ be a random variable with $X\sim\textnormal{Binomial}(m,p)$
or $X\sim\textnormal{Hypergeometric}(pN,N,m)$. Then 
\[
P(\left|X-pm\right|>\varepsilon m)\leq2e^{-2\varepsilon^{2}m}.
\]
\end{prop}

\begin{lem}
\label{lem:randomlike-unordered}Let $n\rightarrow\infty$ and $\frac{m}{n}\rightarrow\infty$.
For any fixed $\varepsilon>0$, for every hypergraph $H$ in $\Hanm$,
with the exception of $o\!\left({{n \choose a} \choose m}\right)$
hypergraphs, the following holds: \newsavebox{\unorderedstatement}\savebox{\unorderedstatement}{\parbox{0.9\textwidth}{For
any coloring $C$ of the vertex set $[n]$ with $b$ colors, and any
$a$-element multiset of colors $T$, the number of hyperedges of
$H$ whose color multiset is $T$ is $\left(p^{C}(T)\pm\varepsilon\right)m$.}}
\begin{equation}
\usebox{\unorderedstatement}\label{eq:unordered-statement}
\end{equation}

(Note: In this paper, whenever we write $X=Y\pm\varepsilon$, we mean
$X\in[Y-\varepsilon,Y+\varepsilon]$.)
\end{lem}

\begin{proof}
Let $u$ be the number of hypergraphs in $\Hanm$ for which \eqref{eq:unordered-statement}
does not hold. Corresponding to each such hypergraph $H$ there is
at least one $b$-coloring $C$ of its vertices, and a multiset of
colors $T$, such that \eqref{eq:unordered-statement} does not hold
for $C$ and $T$. Therefore 
\[
u\leq\left|\left\{ \left(H,C,T\right):H\in\Hanm\textnormal{, \eqref{eq:unordered-statement} does not hold for \ensuremath{H}, \ensuremath{C} and \ensuremath{T}}\right\} \right|.
\]

The number of $b$-colorings of $n$ vertices with $b$ fixed colors
is $\left|\C\right|=b^{n}$. The number of multisets of $a$ elements
of $b$ colors is $\binom{a+b-1}{a}$. Therefore 
\[
u\leq b^{n}\binom{a+b-1}{a}\max_{\substack{\textnormal{coloring }C\\
\textnormal{multiset of colors }T
}
}\left|\left\{ \begin{gathered}H\in\Hanm:\\
\textnormal{\eqref{eq:unordered-statement} does not hold for \ensuremath{H}, \ensuremath{C} and }T
\end{gathered}
\right\} \right|.
\]

Fix a $b$-coloring $C$ and a multiset of colors $T$. A hypergraph
$H\in\Hanm$ consists of $m$ hyperedges, out of $\binom{n}{a}$ possibilities.
Out of all possible hyperedges, $p^{C}(T)\binom{n}{a}$ have $T$
as their color multiset. So $\left|\left\{ e\in H:c^{C}(e)=T\right\} \right|\sim\textnormal{Hypergeometric}\!\left(p^{C}(T)\binom{n}{a},\binom{n}{a},m\right)$.
\eqref{eq:unordered-statement} fails to hold for $H$, $C$ and $T$
if 
\[
\Bigl|\left|\left\{ e\in H:c^{C}(e)=T\right\} \right|-p^{C}(T)m\Bigr|>\varepsilon m.
\]
 By the tail bound for the hypergeometric distribution in Proposition
\ref{prop:tail}, the number of hypergraphs $H\in\Hanm$ for which
this holds is at most 
\[
{{n \choose a} \choose m}\cdot2e^{-2\varepsilon^{2}m},\textnormal{ so}
\]
\[
u\leq2b^{n}\binom{a+b-1}{a}e^{-2\varepsilon^{2}m}{{n \choose a} \choose m}=o\!\left({{n \choose a} \choose m}\right)
\]
 as $\frac{m}{n}\rightarrow\infty$. 
\end{proof}
The following lemma is a corollary of Lemma \ref{lem:randomlike-unordered}.
\begin{lem}
\label{lem:subhypergraph}Let $\T$ be a family of multisets of $a$
elements which are in $[b]$. Let $n\rightarrow\infty$ and $\frac{m}{n}\rightarrow\infty$.
For a $b$-coloring of $n$ vertices $C$, let $p^{C}(\T)=\sum_{T\in\T}p^{C}(T)$
(that is, the probability that the color multiset of a random hyperedge
of the complete $a$-uniform hypergraph is in $\T$); and let $C_{M}$
be a $b$-coloring for which $p^{C}(\T)$ takes its maximum. For a
hypergraph $H\in\Hanm$, let $q(H)$ be the number of hyperedges in
the biggest subhypergraph of $H$ which is colorable in such a way
that the color multiset of every hyperedge of $H$ is in $\T$. For
any fixed $\varepsilon>0$, for every hypergraph $H$ in $\Hanm$,
with the exception of $o\!\left({{n \choose a} \choose m}\right)$
hypergraphs, 
\begin{equation}
q(H)\leq p^{C_{M}}(\T)m\left(1+\varepsilon\right).\label{eq:subhypergraph-statement}
\end{equation}
\end{lem}

\begin{proof}
If $\T=\emptyset$, then $q(H)=0$ for any $H$. From now we assume
that $\T\ne\emptyset$. We show that we may also assume that $p^{C_{M}}(\T)>\frac{\left|\T\right|}{2b^{a}}$
when $n$ is sufficiently large. Let $T\in\T$, and let $C_{E}$ be
a $b$-coloring in which every color class has a size $\approx\frac{n}{b}$.
Then, by Proposition~\ref{prop:p-asymptotic}, asymptotically 
\[
p^{C_{E}}(T)=\frac{a!}{b^{a}\prod_{j=1}^{b}I_{T}(j)!}\geq\frac{1}{b^{a}}.
\]
Since $p^{C_{M}}(\T)\geq p^{C_{E}}(\T)=\sum_{T\in\T}p^{C_{E}}(T)$,
for sufficiently large $n$, $p^{C_{M}}(\T)\geq\frac{\left|\T\right|}{2b^{a}}$.

An equivalent definition of the function $q$ is 
\[
q(H)=\max_{b\textnormal{-coloring }C}\left|\left\{ e\in H:c^{C}(e)\in\T\right\} \right|.
\]
 We use Lemma \ref{lem:randomlike-unordered} with $\frac{\varepsilon}{2b^{a}}$
in place of $\varepsilon$. For almost every hypergraph $H\in\Hanm$,
for every coloring $C$, 
\begin{multline*}
\left|\left\{ e\in H:c^{C}(e)\in\T\right\} \right|=\sum_{T\in\T}\left|\left\{ e\in H:c^{C}(e)=T\right\} \right|\leq\sum_{T\in\T}\left(p^{C}(T)+\frac{\varepsilon}{2b^{a}}\right)m\\
=\left(p^{C}(\T)+\frac{\varepsilon\left|\T\right|}{2b^{a}}\right)m\leq\left(p^{C_{M}}(\T)+\frac{\varepsilon\left|\T\right|}{2b^{a}}\right)m\leq p^{C_{M}}(\T)m\left(1+\varepsilon\right)
\end{multline*}
 using that $p^{C_{M}}(\T)\geq\frac{\left|\T\right|}{2b^{a}}$.
\end{proof}
We define an \emph{oriented hypergraph} as a set of ordered sequences
without repetition (called hyperedges) over a vertex set, such that
two hyperedges are not allowed to differ only in their order. (The
order of the vertices on different hyperedges is independent of each
other.) An oriented hypergraph is thus equivalent to a hypergraph
along with a total order on the vertices of each hyperedge. Let $\Oanm$
denote the family of all $a$-uniform oriented hypergraphs with $n$
vertices and $m$ hyperedges. (Note that other meanings of the term
``oriented hypergraph'' exist in the literature.)

Let $C:[n]\rightarrow[b]$ be a coloring of the vertex set $[n]$
with $b$ colors ($b\geq2$). We call the \emph{color sequence (with
respect to $C$)} of an $a$-tuple of vertices $e=(v_{1},\ldots,v_{a})$
the sequence $c^{C}(e)=\left(C(v_{1}),\ldots,C(v_{a})\right)$. If
we choose a random $a$-tuple of the vertex set $V$ without repetition,
the probability that its color sequence is a given sequence of colors
$s=(s_{1},\ldots,s_{a})$ is 
\[
\frac{1}{\binom{n}{a}a!}\prod_{j=1}^{b}\frac{n_{j}^{C}!}{\left(n_{j}^{C}-\left|\left\{ i\in[a]:s_{i}=j\right\} \right|\right)!}\sim\prod_{i=1}^{a}\frac{n_{c_{i}}^{C}}{n}
\]
 if $n\rightarrow\infty$.

The following lemma is a variant of Lemma \ref{lem:randomlike-unordered}
for oriented hypergraphs.
\begin{lem}
\label{lem:randomlike-ordered}Let $n\rightarrow\infty$ and $\frac{m}{n}\rightarrow\infty$.
For any fixed $\varepsilon>0$, for every oriented hypergraph $O$
in $\Oanm$, with the exception of $o\!\left(\left|\Oanm\right|\right)$
hypergraphs, the following holds: \newsavebox{\orderedstatement}\savebox{\orderedstatement}{\parbox{0.9\textwidth}{For
any coloring $C$ of the vertex set $[n]$ with $b$ colors, and any
$a$-tuple of colors $s$, the number of hyperedges of $O$ whose
color sequence is $s$ is $\left(\prod_{i=1}^{a}\frac{n_{s_{i}}^{C}}{n}\pm\varepsilon\right)\cdot m$.}}
\begin{equation}
\usebox{\orderedstatement}\label{eq:ordered-statement}
\end{equation}
\end{lem}

\begin{proof}
We use Lemma \ref{lem:randomlike-unordered} with $\frac{\varepsilon}{4}$
in the place of $\varepsilon$, i.e.\ that \eqref{eq:unordered-statement}
holds (with $\frac{\varepsilon}{4}$) for almost every hypergraph
$H\in\Hanm$. In every hypergraph in $\Hanm$, the hyperedges can
be ordered in the same number of ways: $(a!)^{m}$. So for almost
every $O\in\Oanm$, \eqref{eq:unordered-statement} holds for the
corresponding hypergraph (obtained by forgetting the orders on the
hyperedges).

Let $\Otanm\subset\Oanm$ be the family of oriented hypergraphs for
which \eqref{eq:unordered-statement} holds (forgetting the orders)
with $\frac{\varepsilon}{4}$ in the place of $\varepsilon$. Let
$u$ be the number of oriented hypergraphs in $\Otanm$ for which
\eqref{eq:ordered-statement} does not hold. Corresponding to each
such oriented hypergraph $O\in\Otanm$, there is at least one $b$-coloring
$C$ of its vertices, and an $a$-tuple of colors $s$, such that
\eqref{eq:ordered-statement} does not hold for $C$ and $s$. Therefore
\[
u\leq\left|\left\{ \left(O,C,s\right):O\in\Otanm\textnormal{, \eqref{eq:ordered-statement} does not hold for \ensuremath{O}, \ensuremath{C} and \ensuremath{s}}\right\} \right|.
\]
The number of $b$-colorings of $n$ vertices with $b$ fixed colors
is $\left|\C\right|=b^{n}$. The number of $a$-tuples of $b$ colors
is $b^{a}$. Therefore 
\[
u\leq b^{n+a}\max_{\substack{\textnormal{coloring }C\\
a\textnormal{-tuple of colors }s
}
}\left|\left\{ \begin{gathered}\left(O,C,s\right):O\in\Otanm,\\
\textnormal{\eqref{eq:ordered-statement} does not hold for \ensuremath{O}, \ensuremath{C} and }s
\end{gathered}
\right\} \right|.
\]

Fix a $b$-coloring $C$ and an $a$-tuple of colors $s$. Let $T$
be the multiset consisting of the elements of $s$ with the multiplicity
with which they occur in $s$ (that is, $T$ is $s$ forgetting the
order). If \eqref{eq:unordered-statement} holds for a $H\in\Hanm$
with $\frac{\varepsilon}{4}$, the number of hyperedges whose color
multiset is $T$ is 
\begin{align*}
M_{H} & :=\left(p^{C}(T)\pm\frac{\varepsilon}{4}\right)m=\left(\frac{\prod_{j=1}^{b}\left(n_{j}^{C}\right)^{I_{T}(j)}}{n^{a}}\cdot\frac{a!}{\prod_{j=1}^{b}I_{T}(j)!}\pm\frac{\varepsilon}{2}\right)m\\
 & =\left(\left(\prod_{i=1}^{a}\frac{n_{s_{i}}^{C}}{n}\right)\cdot\frac{a!}{\prod_{j=1}^{b}I_{T}(j)!}\pm\frac{\varepsilon}{2}\right)m
\end{align*}
 using the Proposition \ref{prop:p-asymptotic} for large enough $n$.
(Changing $\frac{\varepsilon}{4}$ to $\frac{\varepsilon}{2}$ accounts
for the fact that Proposition \ref{prop:p-asymptotic} is asymptotic.)
We can obtain an oriented hypergraph from $H$ by ordering its hyperedges
in one of the $a!$ possible ways, independently from each other.
If we take a hyperedge whose color multiset is $T$, some of these
orders yield the color sequence $s$. The number of such orders is
$\prod_{j=1}^{b}I_{T}(j)!$, so if we take a random ordering of a
hyperedge whose color multiset is $T$, the probability that it has
color sequence $s$ is
\[
\frac{\prod_{j=1}^{b}I_{T}(j)!}{a!}.
\]
So if we obtain an oriented hypergraph $O$ by randomly ordering every
hyperedge of $H$, then $\left|\left\{ e\in O:c^{C}(e)=s\right\} \right|\sim\textnormal{Binomial}\!\left(M_{H},\frac{\prod_{j=1}^{b}I_{T}(j)!}{a!}\right)$,
and the expected value of the number of hyperedges whose color sequence
is $s$ is
\[
E_{H}:=\frac{\prod_{j=1}^{b}I_{T}(j)!}{a!}M_{H}=\left(\prod_{i=1}^{a}\frac{n_{s_{i}}^{C}}{n}\pm\frac{\varepsilon}{2}\right)m.
\]
If the number of hyperedges whose color sequence is $s$ is in the
range $\left[E_{H}-\frac{\varepsilon}{2}m,\allowbreak E_{H}+\frac{\varepsilon}{2}m\right]$,
then \eqref{eq:ordered-statement} holds for $O$, $C$ and $s$,
since 
\[
E_{H}\pm\frac{\varepsilon}{2}m=\left(\prod_{i=1}^{a}\frac{n_{s_{i}}^{C}}{n}\pm\varepsilon\right)m.
\]
 We want to bound the probability that in a randomly selected oriented
hypergraph obtained from $H$, the number of hyperedges whose color
sequence is $s$ is not in the range $\left[E_{H}-\frac{\varepsilon}{2}m,E_{H}+\frac{\varepsilon}{2}m\right]=\left[E_{H}-\frac{\varepsilon m}{2M_{H}}M_{H},E_{H}+\frac{\varepsilon m}{2M_{H}}M_{H}\right]$.
By the tail bound for the binomial distribution in Proposition \ref{prop:tail},
this probability is at most 
\[
2\cdot e^{-2\left((\varepsilon m)/(2M_{H})\right)^{2}M_{H}}=2e^{-\left(\varepsilon^{2}/2\right)\cdot\left(m/M_{H}\right)\cdot m}\leq2e^{-\left(\varepsilon^{2}/2\right)\cdot m},\textnormal{ so}
\]
\begin{comment}
This is tricky: the number of random variables ($M_{H}$) may be small
depending on $C$, but then the coefficient of $M_{H}$ in the tail
bound is also big, so a bad bound is enough.
\end{comment}
\[
u\leq b^{n+a}2e^{-\left(\varepsilon^{2}/2\right)\cdot m}\left|\Otanm\right|=o\!\left(\left|\Oanm\right|\right)
\]
{} as $\frac{m}{n}\rightarrow\infty$.
\end{proof}
\begin{lem}
\label{lem:few short cycles}Let $n\rightarrow\infty$, $k\geq2$
and $m=o\!\left(n^{1+\frac{1}{k}}\right)$. Every hypergraph $H$
in $\Hanm$, with the exception of $o\!\left({{n \choose a} \choose m}\right)$
hypergraphs, has at most $n$ Berge-cycles with $k$ or fewer hyperedges.
\end{lem}

\begin{proof}
A Berge-cycle of length $l$ has $l$ defining vertices, and each
of its $l$ hyperedges contains $a-2$ additional vertices. So the
number of Berge-cycles of length $l$ is less than $n^{(a-1)l}$.
The number of hypergraphs in $\Hanm$ which contain a fixed Berge-cycle
of length $l$ is ${{n \choose a}-l \choose m-l}$, since the $l$
hyperedges of the Berge-cycle can be arbitrarily extended to a hypergraph
of $m$ hyperedges. Therefore the number of pairs $(H,B)$ where $H\in\Hanm$
and $B$ is any Berge-cycle of length $l$ in $H$, is less than 
\[
n^{(a-1)l}{{n \choose a}-l \choose m-l}<n^{(a-1)l}{{n \choose a} \choose m}\left(\frac{m}{{n \choose a}}\right)^{l}=O\!\left(\left(\frac{m}{n}\right)^{l}{{n \choose a} \choose m}\right).
\]
 Let $f_{k}(H)$ denotes the number of Berge-cycles of length $k$
or less in $H$. Using $m=o\!\left(n^{1+\frac{1}{k}}\right)$, we
have
\[
\sum_{H\in\Hanm}f_{k}(H)=\sum_{l=2}^{k}O\!\left(\left(\frac{m}{n}\right)^{l}{{n \choose a} \choose m}\right)=O\!\left(\left(\frac{m}{n}\right)^{k}{{n \choose a} \choose m}\right)=o\!\left(n{{n \choose a} \choose m}\right).
\]
 The number of hypergraphs $H\in\Hanm$ with more than $n$ Berge-cycles
of length $k$ or less is clearly 
\[
\frac{o\!\left(n{{n \choose a} \choose m}\right)}{n}=o\!\left({{n \choose a} \choose m}\right),
\]
 proving Lemma \ref{lem:few short cycles}.
\end{proof}
\begin{prop}
\label{prop:combined}For any $\varepsilon>0$ and $k\ge2$, there
exists an $a$-uniform hypergraph $H$ of girth more than $k$ for
which \eqref{eq:unordered-statement} in Lemma \ref{lem:randomlike-unordered}
and \eqref{eq:subhypergraph-statement} in Lemma \ref{lem:subhypergraph}
hold. There also exists an $a$-uniform oriented hypergraph $O$ of
girth more than $k$ (using the usual meaning of girth, not taking
the orders on the hyperedges into consideration) for which \eqref{eq:ordered-statement}
in Lemma \ref{lem:randomlike-ordered} holds.
\end{prop}

\begin{proof}
Take a sufficiently large $n$, and $m=o\!\left(n^{1+\frac{1}{k}}\right)$
but such that $\frac{m}{n}\rightarrow\infty$ as $n\rightarrow\infty$.
Then there is a hypergraph $H\in\Hanm$ such that \eqref{eq:unordered-statement}
in Lemma \ref{lem:randomlike-unordered} holds with $\frac{\varepsilon}{4b^{a}}$
in place of $\varepsilon$, and $H$ contains at most $n$ Berge-cycles
with $k$ or fewer hyperedges (indeed, all but $o\!\left({{n \choose a} \choose m}\right)$
hypergraphs have both properties). Now remove a hyperedge from every
Berge-cycle of length $k$ or smaller in $H$. The resulting hypergraph
$H'$ has $m-n$ hyperedges. Fix any coloring $C$ and an $a$-element
multiset of colors $T$. In $H$, the number of hyperedges whose color
multiset with respect to $C$ is $T$ is $\left(p^{C}(T)\pm\frac{\varepsilon}{4b^{a}}\right)m$.
The number of such hyperedges in $H'$ is at least $\left(p^{C}(T)-\frac{\varepsilon}{4b^{a}}\right)m-n$
and at most $\left(p^{C}(T)+\frac{\varepsilon}{4b^{a}}\right)m$,
so it is in the range $\left(p^{C}(T)\pm\frac{\varepsilon}{2b^{a}}\right)(m-n)$
for big enough $n$ because $\frac{m}{n}\rightarrow\infty$. So \eqref{eq:unordered-statement}
in Lemma~\ref{lem:randomlike-unordered} holds for $H'$, even with
$\frac{\varepsilon}{2b^{a}}$ in the place of $\varepsilon$. From
the proof of Lemma \ref{lem:subhypergraph} it is clear that if \eqref{eq:unordered-statement}
holds with $\frac{\varepsilon}{2b^{a}}$, then \eqref{eq:subhypergraph-statement}
holds.

In every hypergraph in $\Hanm$, the hyperedges can be ordered in
the same number of ways, so Lemma \ref{lem:few short cycles} holds
for oriented hypergraphs too. The proof in the previous paragraph
works similarly for oriented hypergraphs, proving the existence of
$O$.
\end{proof}
Now we use Proposition \ref{prop:combined} to prove Theorem \ref{thm:high girth small p}.
\begin{boldproof}[Proof of Theorem \ref{thm:high girth small p}]
By Proposition \ref{prop:combined}, there is an $a$-uniform hypergraph
$H$ of girth more than $k$ for which \eqref{eq:subhypergraph-statement}
in Lemma \ref{lem:subhypergraph} holds. We use \eqref{eq:subhypergraph-statement}
with $\T$ consisting of those multisets which contain at least two
different colors, and with $\frac{\varepsilon}{2}$ in the place of
$\varepsilon$. With the notation of Lemma \ref{lem:subhypergraph},
\begin{align*}
q(H) & <p^{C_{M}}(\T)m\left(1+\frac{\varepsilon}{2}\right)=\left(\sum_{T\in\T}p^{C_{M}}(T)\right)m\left(1+\frac{\varepsilon}{2}\right)\\
 & =\left(1-\sum_{j=1}^{b}p^{C_{M}}\Bigl(\Bigl\{\overbrace{j,j,\ldots,j}^{a}\Bigr\}\Bigr)\right)m\left(1+\frac{\varepsilon}{2}\right)\leq\left(1-\sum_{j=1}^{b}\left(\frac{n_{j}^{C_{M}}}{n}\right)^{a}\right)m\left(1+\varepsilon\right)
\end{align*}
 using the asymptotic Proposition \ref{prop:p-asymptotic} for large
enough $n$. $\sum_{j=1}^{b}n_{j}^{C_{M}}=n$, and using the power
mean inequality we get that 
\[
\left(\frac{1}{b}\sum_{j=1}^{b}\left(\frac{n_{j}^{C_{M}}}{n}\right)^{a}\right)^{\frac{1}{a}}\geq\frac{1}{b}.
\]
 So $\sum_{j=1}^{b}\left(\frac{n_{j}^{C_{M}}}{n}\right)^{a}\geq\frac{1}{b^{a-1}}$,
which implies the statement.
\end{boldproof}
We show another example application of Lemma \ref{lem:subhypergraph}
and Proposition \ref{prop:combined}. A $b$-coloring of the vertices
of a hypergraph is called a \emph{rainbow (or strong) coloring} if
all the vertices have different colors in every hyperedge. (For $a$-uniform
hypergraphs, this is only possible if $a\leq b$.)
\begin{prop}
\label{prop:rainbow}Let $n\rightarrow\infty$ and $\frac{m}{n}\rightarrow\infty$.
For any fixed $\varepsilon>0$ and integers $2\leq a\leq b$, every
hypergraph $H$ in $\Hanm$, with the exception of $o\!\left({{n \choose a} \choose m}\right)$
hypergraphs, contains no subhypergraph that is rainbow colorable with
$b$ colors with more than $\binom{b}{a}\frac{a!}{b^{a}}e(H)\left(1+\varepsilon\right)$
hyperedges. Furthermore, for any $\varepsilon>0$ and integers $k\ge2$
and $2\leq a\leq b$, there exists an $a$-uniform hypergraph $H$
of girth more than $k$ which does not contain a subhypergraph that
is rainbow colorable with $b$ colors with more than $\binom{b}{a}\frac{a!}{b^{a}}e(H)\left(1+\varepsilon\right)$
hyperedges.
\end{prop}

\begin{proof}
A hypergraph coloring is a rainbow coloring if the color multiset
of every hyperedge is a conventional set (i.e., every color appears
at most once in the multiset). Let $\T={[b] \choose a}$. We will
prove that if \eqref{eq:subhypergraph-statement} in Lemma \ref{lem:subhypergraph}
holds for a hypergraph $H$ with this $\T$ and with $\frac{\varepsilon}{2}$
in the place of $\varepsilon$, then it does not contain a subhypergraph
that is rainbow colorable with $b$ colors with more than $\binom{b}{a}\frac{a!}{b^{a}}e(H)\left(1+\varepsilon\right)$
hyperedges. The first statement of the proposition then follows directly
from Lemma \ref{lem:subhypergraph}, while the second statement follows
from Lemma \ref{prop:combined}.

With the notation of Lemma \ref{lem:subhypergraph}, and using the
asymptotic Proposition \ref{prop:p-asymptotic} for large enough $n$,
\begin{align}
q(H) & <p^{C_{M}}(\T)m\left(1+\frac{\varepsilon}{2}\right)=\left(\sum_{T\in\T}p^{C_{M}}(T)\right)m\left(1+\frac{\varepsilon}{2}\right)\nonumber \\
 & \leq\left(\sum_{T\in\T}\prod_{j\in T}\frac{n_{j}^{C_{M}}}{n}\right)a!m\left(1+\varepsilon\right).\label{eq:rainbow}
\end{align}
We claim that, under the assumption that $\sum_{j=1}^{b}n_{j}^{C_{M}}=n$,
\eqref{eq:rainbow} takes its maximum when $n_{1}^{C_{M}}=\ldots=n_{b}^{C_{M}}=\frac{n}{b}$.
Let us assume that the $n_{j}^{C_{M}}$'s are not all equal \textendash{}
then there is a $j_{1}$ and $j_{2}$ such that $n_{j_{1}}^{C_{M}}<\frac{n}{b}<n_{j_{2}}^{C_{M}}$.
Rewriting the first factor in \eqref{eq:rainbow}, we have 
\begin{align*}
\sum_{T\in\T}\prod_{j\in T}\frac{n_{j}^{C_{M}}}{n} & =\overbrace{\sum_{\tilde{T}\in\binom{[b]\setminus\{j_{1},j_{2}\}}{a}}\prod_{j\in\tilde{T}}\frac{n_{j}^{C_{M}}}{n}}^{\left|T\cap\{j_{1},j_{2}\}\right|=0}+\overbrace{\left(n_{j_{1}}^{C_{M}}+n_{j_{2}}^{C_{M}}\right)\sum_{\tilde{T}\in\binom{[b]\setminus\{j_{1},j_{2}\}}{a-1}}\prod_{j\in\tilde{T}}\frac{n_{j}^{C_{M}}}{n}}^{\left|T\cap\{j_{1},j_{2}\}\right|=1}\\
 & {}+\overbrace{n_{j_{1}}^{C_{M}}n_{j_{2}}^{C_{M}}\sum_{\tilde{T}\in\binom{[b]\setminus\{j_{1},j_{2}\}}{a-2}}\prod_{j\in\tilde{T}}\frac{n_{j}^{C_{M}}}{n}}^{\left|T\cap\{j_{1},j_{2}\}\right|=2}.
\end{align*}
 If we replace $n_{j_{1}}^{C_{M}}$ with $\frac{n}{b}$, and $n_{j_{2}}^{C_{M}}$
with $n_{j_{2}}^{C_{M}}-\frac{n}{b}+n_{j_{1}}^{C_{M}}$, \eqref{eq:rainbow}
does not decrease: the first two terms do not change, while in the
third term, $n_{j_{1}}^{C_{M}}n_{j_{2}}^{C_{M}}$ is replaced by $\frac{n}{b}\left(n_{j_{2}}^{C_{M}}-\frac{n}{b}+n_{j_{1}}^{C_{M}}\right)=n_{j_{1}}^{C_{M}}n_{j_{2}}^{C_{M}}+\left(n_{j_{2}}^{C_{M}}-\frac{n}{b}\right)\left(\frac{n}{b}-n_{j_{1}}^{C_{M}}\right)>n_{j_{1}}^{C_{M}}n_{j_{2}}^{C_{M}}$.
Repeating this step, we can increase the number of $n_{j}^{C_{M}}$'s
which equal $\frac{n}{b}$ without decreasing \eqref{eq:rainbow},
until all of them equal $\frac{n}{b}$.

So 
\[
q(H)\leq\left(\sum_{T\in\T}\prod_{j\in T}\frac{1}{b}\right)a!m\left(1+\varepsilon\right)=\binom{b}{a}\frac{a!}{b^{a}}m\left(1+\varepsilon\right).\qedhere
\]
\end{proof}

\section{\label{sec:Proof-of-MAINTheorem}Subgraphs of $C_{2k}$-free graphs
\textendash{} Proof of Theorems~\ref{thm:max_c} and \ref{thm:C_4-only}}
\begin{boldproof}[Proof of Theorem \ref{thm:max_c}]
Fix $\varepsilon>0$. By Theorem \ref{thm:high girth small p}, there
exists a $2k-1$-uniform hypergraph $H$ with girth more than $2k$
which does not contain a 2-colorable subhypergraph having more than
$\left(1-\frac{1}{2^{2k-2}}\right)e(H)\left(1+\varepsilon\right)$
hyperedges. We produce a graph $G_{H}$ from the hypergraph $H$ by
replacing each hyperedge of $H$ with a complete graph (i.e.\ a clique)
on $2k-1$ vertices. We refer to these complete graphs as $2k-1$-cliques.
It is easy to check that the resulting graph $G_{H}$ is $C_{2k}$-free. 

Notice that since the girth of $H$ is more than $2k\ge4$, no two
hyperedges of $H$ intersect in more than $1$ vertex. Therefore,
the $2k-1$-cliques of $G_{H}$ are edge-disjoint, and by definition
every edge of $G_{H}$ is in some $2k-1$-clique. We show that $G_{H}$
does not have a bipartite subgraph with girth more than $2k$ which
has more than $\left(1-\frac{1}{2^{2k-2}}\right)\frac{2k-2}{\binom{2k-1}{2}}e(G_{H})\left(1+\varepsilon\right)=\left(1-\frac{1}{2^{2k-2}}\right)\frac{2}{2k-1}e(G_{H})\left(1+\varepsilon\right)$
edges. Assume that $B$ is a bipartite subgraph of $G_{H}$ with girth
more than $2k$. Notice that any set of more than $2k-2$ edges from
a clique on $2k-1$ vertices must contain a cycle of length at most
$2k-1$. Therefore $B$ can contain at most $2k-2$ edges from each
$2k-1$-clique of $G_{H}$. Furthermore, since $B$ is bipartite,
there is a $2$-coloring of the vertices so that the edges of $B$
are properly colored. If an edge of $B$ is contained in a $2k-1$-clique
of $G_{H}$, then the corresponding hyperedge of $H$ contains two
vertices with different colors in this $2$-coloring. By our assumption
on $H$, at most $\left(1-\frac{1}{2^{2k-2}}\right)(1+\varepsilon)$
fraction of the hyperedges are not monochromatic in this 2-coloring
of the vertices. So $B$ has at most $\left(1-\frac{1}{2^{2k-2}}\right)(2k-2)e(H)\left(1+\varepsilon\right)$
edges. Since $e(G_{H})=\binom{2k-1}{2}e(H)$, $B$ has at most $\left(1-\frac{1}{2^{2k-2}}\right)\frac{2}{2k-1}e(G_{H})\left(1+\varepsilon\right)$
edges, as desired.
\end{boldproof}
In the proof of Theorem \ref{thm:C_4-only}, we use the following
proposition. For a proof, see the proof of Proposition 5 in \cite{Kuhn:2005}.
(Note that the bound can be attained when $w\geq\binom{u}{2}$.)
\begin{prop}
\label{prop:complete-bipartite}In the complete bipartite graph $K_{u,w}$,
a $C_{4}$-free subgraph has at most $w+\binom{u}{2}$ edges.
\end{prop}

\begin{boldproof}[Proof of Theorem \ref{thm:C_4-only}]
Let $l$ be a large integer. By Proposition \ref{prop:combined},
there exists a $k-1+l$-uniform oriented hypergraph $O$ with girth
more than $2k$ for which \eqref{eq:ordered-statement} in Lemma \ref{lem:randomlike-ordered}
holds with $\frac{\varepsilon}{24\cdot2^{k-1+l}}$ in place of $\varepsilon$.
Let $n$ be the number of vertices of $O$. We produce a graph $G_{O}$
from the oriented hypergraph $O$ by replacing each hyperedge of $O$
with a copy of $K_{k-1,m}$ the following way: in a hyperedge $(v_{1},\ldots,v_{k-1+l})$,
we connect every vertex in $\{v_{1},\ldots,v_{k-1}\}$ with every
vertex in $\{v_{k},\ldots,v_{k-1+l}\}$ with an edge. The resulting
graph $G_{O}$ is $C_{2k}$-free.\setlength\emergencystretch{\hsize}

Since the girth of $O$ is more than $2k\geq4$, no two hyperedges
of $O$ intersect in more than 1 vertex. Therefore the copies of $K_{k-1,l}$
in $G_{O}$ are edge-disjoint, and by definition every edge of $G_{O}$
is in one of the copies of $K_{k-1,l}$. We show that $G_{O}$ does
not have a bipartite and $C_{4}$-free subgraph which has more than
$\left(1-\frac{1}{2^{k-1}}\right)\frac{1}{k-1}e(G_{O})(1+\varepsilon)$
edges. Assume that $B$ is a bipartite and $C_{4}$-free subgraph
of $G_{O}$, its classes being $pn$ red vertices and $(1-p)n$ blue
vertices. Now consider a random hyperedge $e=(v_{1},\ldots,v_{k-1},v_{k},\ldots,v_{k-1+l})$
of $O$. How many edges of $B$ can there be between the vertices
of $e$? Each such edge has a red and a blue endpoint; also, each
such edge has an endpoint in $\{v_{1},\ldots,v_{k-1}\}$ and an endpoint
in $\{v_{k},\ldots,v_{k-1+l}\}$. Let $u$ and $w$ be the number
of red vertices among $\{v_{1},\ldots,v_{k-1}\}$ and $\{v_{k},\ldots,v_{k-1+l}\}$
respectively. The restriction of $B$ to the vertices of $e$ (which
we will denote \global\long\def\Be{B|_{e}}
$\Be$) is thus a $C_{4}$-free subgraph of the union of a $K_{u,l-w}$
and a $K_{k-1-u,w}$ on disjoint vertex sets. We have three possibilities:
\begin{itemize}
\item \setlength\emergencystretch{0em}$u\notin\{0,k-1\}$. Then, by Proposition
\ref{prop:complete-bipartite}, $\Be$ consists of at most $l-w+{u \choose 2}+w+\binom{k-1-u}{2}<l+\binom{k-1}{2}$
edges.
\item $u=k-1$. Then $K_{k-1-u,w}$ is degenerate (as $k-1-u=0$), and $\Be$
has at most $l-w+{k-1 \choose 2}$ edges.
\item $u=0$. Then $K_{u,l-w}$ is degenerate, and $\Be$ has at most $w+\binom{k-1}{2}=l+\binom{k-1}{2}-(l-w)$
edges.
\end{itemize}
Let $(C_{1},\ldots,C_{k-1+l})$ be the color sequence of $e$ (with
$C_{i}\in\{\red,\blue\}$). For any color sequence $(c_{1},\ldots,\allowbreak c_{k-1+l})$
(with $c_{i}\in\{\red,\blue\}$), the probability that $(C_{1},\ldots,C_{k-1+l})=(c_{1},\ldots,c_{k-1+l})$
is $p^{\left|\left\{ i:c_{i}=\red\right\} \right|}(1-p)^{\left|\left\{ i:c_{i}=\blue\right\} \right|}\pm\frac{\varepsilon}{24\cdot2^{k-1+l}}$
since \eqref{eq:ordered-statement} in Lemma \ref{lem:randomlike-ordered}
holds for $O$ with $\frac{\varepsilon}{24\cdot2^{k-1+l}}$. (Note
that $e$ was chosen as a random hyperedge of $O$.) Let $\tilde{C}_{1},\ldots,\tilde{C}_{k-1+l}$
be independent random variables which take the value ``$\red$''
with probability $p$ and the value ``$\blue$'' with probability
$1-p$. Let $f(C_{1},\ldots,C_{k-1+l})$ be a real valued function
of a color sequence. We claim that
\begin{equation}
\bigl|\E(f(C_{1},\ldots,C_{k-1+l}))-\E(f(\tilde{C}_{1},\ldots,\tilde{C}_{k-1+l}))\bigr|\leq\frac{\varepsilon}{24}\max\left|f\right|.\label{eq:randomlike}
\end{equation}
 Indeed, 
\[
\E(f(\tilde{C}_{1},\ldots,\tilde{C}_{k-1+l}))=\sum_{\substack{(c_{1},\ldots,c_{k-1+l})\\
{}\in\{\red,\blue\}^{k-1+l}
}
}p^{\left|\left\{ i:c_{i}=\red\right\} \right|}(1-p)^{\left|\left\{ i:c_{i}=\blue\right\} \right|}f(c_{1},\ldots,c_{k-1+l}),\textnormal{ and}
\]
\begin{gather*}
\E(f(C_{1},\ldots,C_{k-1+l}))=\sum_{\substack{(c_{1},\ldots,c_{k-1+l})\\
{}\in\{\red,\blue\}^{k-1+l}
}
}\left(p^{\left|\left\{ i:c_{i}=\red\right\} \right|}(1-p)^{\left|\left\{ i:c_{i}=\blue\right\} \right|}\pm\frac{\varepsilon}{24\cdot2^{k-1+l}}\right)\cdot\\
\cdot f(c_{1},\ldots,c_{k-1+l})=\E(f(\tilde{C}_{1},\ldots,\tilde{C}_{k-1+l}))\\
{}+\sum_{\substack{(c_{1},\ldots,c_{k-1+l})\\
{}\in\{\red,\blue\}^{k-1+l}
}
}\left(\pm\frac{\varepsilon}{24\cdot2^{k-1+l}}\right)f(c_{1},\ldots,c_{k-1+l})\\
=\E(f(\tilde{C}_{1},\ldots,\tilde{C}_{k-1+l}))\pm\frac{\varepsilon}{24}\max\left|f\right|.
\end{gather*}

Using \eqref{eq:randomlike} with $f(C_{1},\ldots,C_{k-1+l})\allowbreak=\begin{cases}
1 & \textnormal{if }C_{1}=\ldots=C_{k-1}=\red\\
0 & \textnormal{otherwise}
\end{cases}$, we have $P(u=k-1)\allowbreak=\E(I_{u=k-1})=p^{k-1}\pm\frac{\varepsilon}{24}$;
with $f(C_{1},\ldots,C_{k-1+l})\allowbreak=\begin{cases}
1 & \textnormal{if }C_{1}=\ldots=C_{k-1}=\blue\\
0 & \textnormal{otherwise}
\end{cases}$, we have $P(u=0)\allowbreak=\E(I_{u=0})=(1-p)^{k-1}\pm\frac{\varepsilon}{24}$;
and with $f(C_{1},\ldots,C_{k-1+l})\allowbreak=\left|\{i\in\{k,\ldots,k-1+l\}:C_{i}=\red\}\right|$,
we have $E(w)=pl\pm\frac{\varepsilon}{24}l$. So 
\begin{align*}
\E(e(\Be)) & =P(u\notin\{0,k-1\})\left(l+\binom{k-1}{2}\right)+P(u=k-1)E\!\left(l-w+{k-1 \choose 2}\right)\\
 & {}+P(u=0)E\!\left(l+\binom{k-1}{2}-(l-w)\right)\\
 & \leq l+\binom{k-1}{2}-\left(p^{k-1}\pm\frac{\varepsilon}{24}\right)\left(p\pm\frac{\varepsilon}{24}\right)l-\left((1-p)^{k-1}\pm\frac{\varepsilon}{24}\right)\left(1-p\pm\frac{\varepsilon}{24}\right)l\\
 & \leq l+\binom{k-1}{2}-p^{k}l-(1-p)^{k}l+\frac{\varepsilon}{4}l\leq\left(1-\frac{1}{2^{k-1}}\right)l+\binom{k-1}{2}+\frac{\varepsilon}{4}l
\end{align*}
 assuming $\varepsilon\leq1$.

That is, if $O$ has $m$ hyperedges, $e(B)=\left(\left(1-\frac{1}{2^{k-1}}\right)l+\binom{k-1}{2}+\frac{\varepsilon}{4}l\right)m$,
while $e(G_{O})=m(k-1)l$. Let $l\geq\frac{k(k-2)}{\varepsilon}$,
then 
\begin{align*}
e(B) & \leq\left(\left(1-\frac{1}{2^{k-1}}\right)\frac{1}{k-1}+\frac{k-2}{2l}+\frac{\varepsilon}{4(k-1)}\right)e(G_{O})\\
 & \leq\left(\left(1-\frac{1}{2^{k-1}}\right)\frac{1}{k-1}+\frac{\varepsilon}{k}\right)e(G_{O})\leq\left(1-\frac{1}{2^{k-1}}\right)\frac{1}{k-1}e(G_{O})(1+\varepsilon).\qedhere
\end{align*}
\end{boldproof}

\section{\label{sec:pastings}Pasting $C_{6}$'s to produce a $C_{8}$-free
graph}

We will make use of the following proposition of Nešet\v{r}il and
Rödl \cite{nevsetvril1978probabilistic} in the second example, and
in the general version of the first example.
\begin{prop}[Nešet\v{r}il, Rödl \cite{nevsetvril1978probabilistic}]
\label{Prob_method} For any positive integers $r\ge2$ and $s\ge3$,
there exists an $n_{0}\in\naturals$ such that for any integer $n\ge n_{0}$
there is a $r$-uniform hypergraph with girth at least $s$ and more
than $n^{1+1/s}$ hyperedges. 
\end{prop}

\subsection{First example}

For our construction here we will need a bipartite graph of girth
at least 10 with many edges (and with degree at least 2 in every vertex).
We will derive such a graph from the following construction of Benson
\cite{benson1966minimal}.
\begin{thm}[Benson \cite{benson1966minimal}]
Let $q$ be an odd prime power. There is a $(q+1)$-regular, bipartite,
girth 12 graph $Q$ with $2(q^{5}+q^{4}+q^{3}+q^{2}+q+1)$ vertices. 
\end{thm}

First, let us notice that since $Q$ is a regular bipartite graph,
it has color classes of equal size. Moreover, we may assume that $Q$
is connected, for otherwise we may add some edges to make it connected
without creating cycles. So we have the following corollary.
\begin{cor}
\label{girth 10} There exists a connected bipartite graph of girth
at least 10 with $n/2$ vertices in each color class such that every
vertex has degree at least $(n/2)^{1/5}(1-o(1))$. (So it contains
at least $(1-o(1))(n/2)^{6/5}$ edges.) 
\end{cor}

\begin{comment}
Btw we need degrees>=2 so that every edge is in a C6 
\end{comment}
\begin{thm}
\label{pasting2} There exists a $C_{8}$-free graph $G$ on $4n$
vertices with average degree at least $4n^{1/5}$ which is pasted
together from $C_{6}$'s. 
\end{thm}

To prove Theorem \ref{pasting2}, let us take a connected, bipartite
graph $G_{1}$ of girth at least 10 on $2n$ vertices such that each
vertex has degree at least $n^{1/5}(1-o(1))$ (and having $n$ vertices
in each color class). The existence of such a graph is guaranteed
by Corollary \ref{girth 10}. Let $a_{1},a_{2},\ldots,a_{n}$ and
$b_{1},b_{2},\ldots,b_{n}$ be the two color classes of $G_{1}$.
Now let $G_{2}$ be a copy of $G_{1}$ with vertices $a'_{1},a'_{2},\ldots,a'_{n}$
and $b'_{1},b'_{2},\ldots,b'_{n}$ and edge set $E(G_{2})=\{a'_{i}b'_{j}\mid a_{i}b_{j}\in E(G_{1})\}$.
Finally, the graph $G$ is defined to have the vertex set $V(G)=V(G_{1})\cup V(G_{2})$
and the edge set $E(G)=E(G_{1})\cup E(G_{2})\cup\{b_{i}b'_{i}\mid1\le i\le n\}$
(see Figure \ref{first pasting}). So $G$ has $4n$ vertices and
$2n^{6/5}(1-o(1))+n$ edges.

\begin{figure}
\begin{centering}
\begin{tikzpicture}[line cap=round,line join=round,>=triangle 45,x=1.0cm,y=1.0cm]
\begin{comment}
\clip(5.56,6.48) rectangle (5.62,6.48);
\end{comment}
\draw (4,4)-- (3,2);
\draw (4,4)-- (4,2);
\draw (4,4)-- (5,2);
\draw (4,4)-- (2,2);
\draw (6,4)-- (5,2);
\draw (6,4)-- (6,2);
\draw (6,4)-- (7,2);
\draw (6,4)-- (8,2);
\draw (2,2)-- (2,0);
\draw (1,2)-- (1,0);
\draw (3,2)-- (3,0);
\draw (4,2)-- (4,0);
\draw (5,2)-- (5,0);
\draw (6,2)-- (6,0);
\draw (7,2)-- (7,0);
\draw (8,2)-- (8,0);
\draw (9,2)-- (9,0);
\draw (10,2)-- (10,0);
\draw (4,-2)-- (4,0);
\draw (4,-2)-- (5,0);
\draw (4,-2)-- (3,0);
\draw (4,-2)-- (2,0);
\draw (6,-2)-- (5,0);
\draw (6,-2)-- (6,0);
\draw (6,-2)-- (7,0);
\draw (6,-2)-- (8,0);
\begin{scriptsize}
\fill [color=black] (1,4) circle (2.5pt);
\draw[color=black] (1.28,4.36) node {$a_1$};
\fill [color=black] (2,4) circle (2.5pt);
\draw[color=black] (2.28,4.36) node {$a_2$};
\fill [color=black] (3,4) circle (2.5pt);
\draw[color=black] (3.28,4.36) node {$a_3$};
\fill [color=black] (4,4) circle (2.5pt);
\draw[color=black] (4.28,4.36) node {$a_4$};
\fill [color=black] (5,4) circle (2.5pt);
\draw[color=black] (5.28,4.36) node {$a_5$};
\fill [color=black] (6,4) circle (2.5pt);
\draw[color=black] (6.28,4.36) node {$a_6$};
\fill [color=black] (7,4) circle (2.5pt);
\fill [color=black] (8,4) circle (2.5pt);
\fill [color=black] (9,4) circle (2.5pt);
\draw[color=black] (9.46,4.36) node {$a_{n-1}$};
\fill [color=black] (1,2) circle (2.5pt);
\draw[color=black] (1.26,2.36) node {$b_1$};
\fill [color=black] (2,2) circle (2.5pt);
\draw[color=black] (2.26,2.36) node {$b_2$};
\fill [color=black] (3,2) circle (2.5pt);
\draw[color=black] (3.26,2.36) node {$b_3$};
\fill [color=black] (4,2) circle (2.5pt);
\draw[color=black] (4.26,2.36) node {$b_4$};
\fill [color=black] (5,2) circle (2.5pt);
\draw[color=black] (5.26,2.36) node {$b_5$};
\fill [color=black] (6,2) circle (2.5pt);
\draw[color=black] (6.26,2.36) node {$b_6$};
\fill [color=black] (7,2) circle (2.5pt);
\fill [color=black] (8,2) circle (2.5pt);
\fill [color=black] (9,2) circle (2.5pt);
\draw[color=black] (9.44,2.36) node {$b_{n-1}$};
\fill [color=black] (1,0) circle (2.5pt);
\draw[color=black] (1.28,0.36) node {$b'_1$};
\fill [color=black] (2,0) circle (2.5pt);
\draw[color=black] (2.28,0.36) node {$b'_2$};
\fill [color=black] (3,0) circle (2.5pt);
\draw[color=black] (3.28,0.36) node {$b'_3$};
\fill [color=black] (4,0) circle (2.5pt);
\draw[color=black] (4.28,0.36) node {$b'_4$};
\fill [color=black] (5,0) circle (2.5pt);
\draw[color=black] (5.28,0.36) node {$b'_5$};
\fill [color=black] (6,0) circle (2.5pt);
\draw[color=black] (6.28,0.36) node {$b'_6$};
\fill [color=black] (7,0) circle (2.5pt);
\fill [color=black] (8,0) circle (2.5pt);
\fill [color=black] (9,0) circle (2.5pt);
\draw[color=black] (9.46,0.36) node {$b'_{n-1}$};
\fill [color=black] (1,-2) circle (2.5pt);
\draw[color=black] (1.3,-1.64) node {$a'_1$};
\fill [color=black] (2,-2) circle (2.5pt);
\draw[color=black] (2.3,-1.64) node {$a'_2$};
\fill [color=black] (3,-2) circle (2.5pt);
\draw[color=black] (3.3,-1.64) node {$a'_3$};
\fill [color=black] (4,-2) circle (2.5pt);
\draw[color=black] (4.3,-1.64) node {$a'_4$};
\fill [color=black] (5,-2) circle (2.5pt);
\draw[color=black] (5.3,-1.64) node {$a'_5$};
\fill [color=black] (6,-2) circle (2.5pt);
\draw[color=black] (6.3,-1.64) node {$a'_6$};
\fill [color=black] (7,-2) circle (2.5pt);
\fill [color=black] (8,-2) circle (2.5pt);
\fill [color=black] (9,-2) circle (2.5pt);
\draw[color=black] (9.48,-1.64) node {$a'_{n-1}$};
\fill [color=black] (10,4) circle (2.5pt);
\draw[color=black] (10.36,4.36) node {$a_{n}$};
\fill [color=black] (10,2) circle (2.5pt);
\draw[color=black] (10.26,2.36) node {$b_n$};
\fill [color=black] (10,0) circle (2.5pt);
\draw[color=black] (10.28,0.36) node {$b'_n$};
\fill [color=black] (10,-2) circle (2.5pt);
\draw[color=black] (10.3,-1.64) node {$a'_n$};
\end{scriptsize}
\end{tikzpicture} 
\par\end{centering}
\caption{First pasting}
\label{first pasting} 
\end{figure}
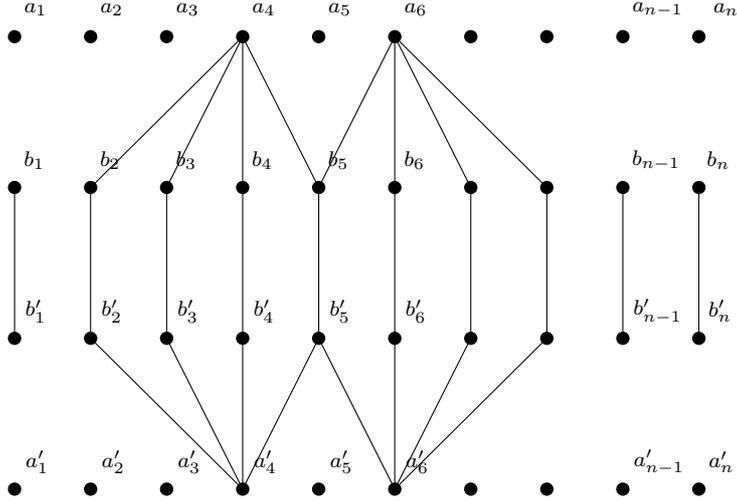

To show that $G$ is pasted together from $C_{6}$'s, we have to show
that every edge is contained in a $C_{6}$, and for any two edges
$e_{1},e_{2}\in E(G)$, there is a sequence $O_{1},O_{2},\ldots,O_{m}$
of $C_{6}$'s in $G$ such that for any $1\le i\le m-1$, $O_{i}$
and $O_{i+1}$ share at least one edge, and $e_{1}$ and $e_{2}$
are edges of $O_{1}$ and $O_{m}$ respectively. The graph can then
be built starting from an arbitrary fixed edge. It is easy to see
that every edge is contained in some $C_{6}$ of the form $a_{i}b_{j}b'_{j}a'_{i}b'_{k}b_{k}$,
and so we can assume that both $e_{1}$ and $e_{2}$ are of the form
$a_{i}b_{j}$. Let $(a_{i_{0}})b_{i_{1}}a_{i_{2}}b_{i_{3}}a_{i_{4}}b_{i_{5}}\ldots a_{i_{t-1}}b_{i_{t}}(a_{i_{t+1}})$
be a path starting with $e_{1}$ and ending with $e_{2}$, with $e_{1}=a_{i_{0}}b_{i_{1}}$
or $e_{1}=b_{i_{1}}a_{i_{2}}$, and $e_{2}=a_{i_{t-1}}b_{i_{t}}$
or $e_{2}=b_{i_{t}}a_{i_{t+1}}$ (such a path exists since $G_{1}$
is connected). Then the path $b'_{i_{1}}a'_{i_{2}}b'_{i_{3}}a'_{i_{4}}b'_{i_{5}}\ldots a'_{i_{t-1}}b'_{i_{t}}$
is contained in $G_{2}$. These two paths together with the edges
$b_{i_{1}}b'_{i_{1}},b_{i_{2}}b'_{i_{2}},\ldots,b_{i_{t}}b'_{i_{t}}$
give the desired sequence of $C_{6}$'s (together with an arbitrary
$C_{6}$ of the form $a_{i_{0}}b_{i_{1}}b'_{i_{1}}a'_{i_{0}}b'_{j}b_{j}$
if $e_{1}=a_{i_{0}}b_{i_{1}}$, and a $C_{6}$ of the form $a_{i_{t+1}}b_{i_{t}}b'_{i_{t}}a'_{i_{t+1}}b'_{k}b_{k}$
if $e_{2}=b_{i_{t}}a_{i_{t+1}}$).

It remains to show that $G$ is $C_{8}$-free. Suppose for a contradiction
that it has a $C_{8}$. Since the graph $G_{1}$ is of girth at least
10, this $C_{8}$ cannot be completely in $G_{1}$ or $G_{2}$. So
it has to contain at least one edge from each of the three sets $E(G_{1})$,
$E(G_{2})$ and $\{b_{i}b'_{i}\mid1\le i\le n\}$. Moreover, it is
easy to see that it contains exactly two edges from $\{b_{i}b'_{i}\mid1\le i\le n\}$,
say $b_{i}b'_{i}$ and $b_{j}b'_{j}$. So there is a path of $q$
edges between $b_{i}$ and $b_{j}$ in $G_{1}$ and a path of $r$
edges between $b'_{i}$ and $b'_{j}$ in $G_{2}$ such that $q+r=6$.
Let these paths be $b_{i}a_{i_{1}}b_{i_{2}}\ldots a_{i_{q-1}}b_{j}$
and $b'_{i}a'_{j_{1}}b'_{j_{2}}\ldots a'_{j_{r-1}}b'$ respectively.
By construction, the second path in $G_{2}$ implies that $G_{1}$
contains the path $b_{i}a_{j_{1}}b_{j_{2}}\ldots a_{j_{r-1}}b_{j}$,
which, together with $b_{i}a_{i_{1}}b_{i_{2}}\ldots a_{i_{q-1}}b_{j}$,
would produce a cycle of length 4 or 6 in $G_{1}$, a contradiction. 
\begin{rem}
We may modify the above construction as described below to find a
pasting of $C_{2l}$'s to produce a $C_{2k}$-free graph $G$ for
any given integers $k>l\ge3$ and having average degree at least $\Omega\!\left(n^{1/(2k+2)}\right)$.
\end{rem}

\begin{proof}
A graph of girth $2k+1$ and having $\Omega\!\left(n^{1+1/(2k+1)}\right)$
edges exists by applying Proposition \ref{Prob_method} with $r=2$.
So it has average degree $\Omega\!\left(n^{1/(2k+1)}\right)$. It
is easy to find a bipartite subgraph of such a graph, with equal color
classes and having a constant fraction of all the edges. Then we can
delete vertices of degree $1$ without decreasing its average degree,
so we can assume it has minimum degree at least $2$, and as usual,
we can assume it is connected, because otherwise we can add edges
without creating a cycle to make it connected. Let $G_{1}$ be this
bipartite, connected graph of girth greater than $2k$ on $2n$ vertices
with average degree $\Omega\!\left(n^{1/(2k+1)}\right)$. Then let
$G_{2}$ be defined in the same way as in the above proof (based on
$G_{1}$). However, now, for each $i$ we connect the vertices $b_{i}\in V(G_{1})$
and $b'_{i}\in V(G_{2})$ by a path containing $l-2$ edges and let
the resulting graph be $G$. Using the same argument as in the above
proof, we can see that this gives a pasting of $C_{2l}$'s and that
$G$ is $C_{2k}$-free.
\end{proof}

\subsection{Second example}

A hypergraph $H$ is \emph{connected} if for any two vertices $u,v\in V(H)$,
there exist hyperedges $h_{i}\in E(H)$, $1\le i\le m$, such that
$u\in h_{1},v\in h_{m}$ and $h_{i}\cap h_{i+1}\not=\emptyset$ for
all $1\le i\le m-1$. %
\begin{comment}
Merge with Berge-cycle definition in the making-bipartite article? 
\end{comment}
\begin{comment}
Yes 
\end{comment}
A minimal collection of such hyperedges is called a path between $u$
and $v$ in $H$. We may assume that the hypergraph provided by Proposition
\ref{Prob_method} is connected, for otherwise we can simply take
a connected component of it containing the biggest number of hyperedges.
\begin{thm}
\label{thm:firstpasting}There exists a $C_{8}$-free graph $G$ on
$2n$ vertices with average degree at least $6\cdot n^{1/9}$ which
is pasted together from $C_{6}$'s. 
\end{thm}

To prove Theorem \ref{thm:firstpasting}, we apply Proposition \ref{Prob_method}
to find a (connected) 3-uniform hypergraph $H_{1}$ on $n$ vertices
with girth at least $9$ and more than $n^{1+1/9}$ hyperedges. Let
$V(H_{1})=\{u_{1},u_{2},\ldots,u_{n}\}$. Replace each vertex $u_{i}\in V(H_{1})$
with a pair of vertices $u_{i},u'_{i}$ so that every hyperedge containing
$u_{i}$ now contains both $u_{i}$ and $u'_{i}$. This produces a
$6$-uniform hypergraph which we denote by $H_{2}$.

\begin{comment}
Let us take the $3$-uniform (connected) hypergraph $H_{1}$ on $n$
vertices with girth at least $9$ and having more than $n^{1+\frac{1}{9}}$
hyperedges. The existence of such a hypergraph is guaranteed by Proposition
\ref{Prob_method}. Let $V(H_{1})=\{u_{1},u_{2},\ldots,u_{n}\}$.
Replace each vertex $u_{i}\in V(H_{1})$ with a pair of vertices $u_{i},u'_{i}$
so that every hyperedge containing $u_{i}$ now contains both $u_{i}$
and $u'_{i}$. This will produce a $6$-uniform hypergraph, say $H_{2}$. 
\end{comment}

Now we construct the desired graph $G$ from $H_{2}$ in the following
fashion. If $\{u_{i},u'_{i},u_{j},u'_{j},\allowbreak u_{k},u'_{k}\}$
is a hyperedge in $H_{2}$ with $1\le i\le j\le k\le n$, then we
add the edges $u_{i}u'_{i},u'_{i}u_{j},u_{j}u'_{j},u'_{j}u_{k},u_{k}u'_{k},u'_{k}u_{i}$
to $E(G)$. We repeat this procedure for every hyperedge of $H_{2}$.
Let us call the edges $u_{i}u'_{i}\in E(G)$ ($1\le i\le n$) \emph{fat}
edges and the rest of the edges of $G$ \emph{thin} edges. 

Note that two fat edges never share a vertex. We claim that a thin
edge cannot be added by two different hyperedges of $H_{2}$. Suppose
by contradiction that $u'_{i}u_{j}$ is a thin edge added by two different
hyperedges $h_{1},h_{2}$ of $H_{2}$. Then since a hyperedge of $H_{2}$
either contains both vertices $u_{r},u'_{r}$ or neither of them for
any given $1\le r\le n$, it follows that $\{u_{i},u'_{i},u_{j},u'_{j}\}\subset h_{1}$
and $\{u_{i},u'_{i},u_{j},u'_{j}\}\subset h_{2}$. Consider the two
hyperedges in $H_{1}$ which correspond to $h_{1}$ and $h_{2}$.
They both contain $u_{i}$ and $u_{j}$; so they intersect in at least
two vertices, which is a contradiction since $H_{1}$ is a linear
hypergraph. (Notice, on the other hand, that a fat edge may have been
added by several hyperedges.) So each hyperedge in $H_{2}$ adds precisely
$3$ new thin edges to $E(G)$. Therefore the number of thin edges
in $G$ is three times the number of hyperedges in $H_{2}$. Since
the number of fat edges is $n$, we have $\abs{E(G)}=3\cdot n^{1+1/9}+n$.
Thus it has the desired average degree.

Since $H_{1}$ is connected, we can construct it by adding hyperedges
one by one, in such a way that each hyperedge intersects one of the
previous hyperedges in at least one vertex. We can construct $H_{2}$
by adding the $C_{6}$'s corresponding to the hyperedges of $H_{1}$
in the same order; this shows that $G$ is pasted together from $C_{6}$'s. 

It only remains to show that $G$ is $C_{8}$-free. We say an edge
is between two edges $e_{1}$, $e_{2}$ if one of its end vertices
is in $e_{1}$ and the other is in $e_{2}$.
\begin{claim}
\label{claim:onethinbetweentwofat}There is at most one thin edge
between any two fat edges of $G$.
\end{claim}

\begin{proof}
Consider any two fat edges $u_{i}u'_{i}$ and $u_{j}u'_{j}$ of $G$.
As noted earlier, any thin edge between them is added by a unique
hyperedge $h$ of $H_{2}$, and $h$ contains all four vertices $u_{i},u'_{i},u_{j},u'_{j}$.
Because of the linearity of $H_{1}$, no hyperedge of $H_{2}$ other
than $h$ may contain all four vertices $u_{i},u'_{i},u_{j},u'_{j}$
Now note that in our procedure, any hyperedge of $H_{2}$ adds at
most one thin edge between any two fat edges contained in it, proving
the claim.
\end{proof}
Now suppose for a contradiction that $G$ contains a $C_{8}$. Since
no two fat edges in $G$ share a vertex, there can be at most four
fat edges in this $C_{8}$. Contract every pair of vertices $u_{i},u'_{i}$
in $G$ into a single vertex $u_{i}$. Then this $C_{8}$ would become
a closed walk $C'$ of length at most $8$ and at least $4$ (only
thin edges remain after contraction). While this closed walk may have
repeated vertices, we show that it cannot have repeated edges (i.e.,
it is actually a circuit). Suppose that after contracting every pair
of vertices $u_{i},u'_{i}$ to $u_{i}$, some two thin edges $xy$
and $zw$ coincide. Then, for some $i$ and $j$, we have $x,z\in\{u_{i},u'_{i}\}$
and $y,w\in\{u_{j},u'_{j}\}$. Between the fat edges $u_{i}u'_{i}$
and $u_{j}u'_{j}$, there are two thin edges (namely $xy$ and $zw$),
contradicting Claim \ref{claim:onethinbetweentwofat}.

The 2-shadow of a hypergraph $H$ is the graph which contains an edge
$uv$ if and only if there is a hyperedge of $H$ which contains $u$
and $v$. $C'$ must be contained in the $2$-shadow of $H_{1}$.
Since $H_{1}$ has girth at least 9, it is not difficult to see that
the only possible length of a circuit in its $2$-shadow that is between
$4$ and $8$ is $6$ and it must be of the form $ab,bc,ce,ed,dc,ca$
(notice that $c$ is a repeated vertex). Therefore, the original $C_{8}$
in $G$ must be contained in the set of edges added by the hyperedges
$\{a,a',b,b',c,c'\}$ and $\{c,c',d,d',e,e'\}$ of $H_{2}$, but this
is impossible as these edges consist of two $C_{6}$'s sharing exactly
one edge.

\section*{Acknowledgments\phantomsection\addcontentsline{toc}{section}{Acknowledgements}}

We are grateful to Marianna Bolla, Ervin Gy\H ori and D\"om\"ot\"or
P\'alv\"olgyi for pointing us to helpful references and for useful
discussions. The authors Methuku and Tompkins were supported by the
National Research, Development and Innovation Office \textendash{}
NKFIH under the grant K116769.

\bibliographystyle{plain}
\bibliography{bibliography1}

\end{document}